\documentclass[a4paper,12pt,draft]{amsart}
\usepackage[margin=30mm]{geometry}
\usepackage{amssymb, amsmath, amsthm, amscd}

\usepackage[T1]{fontenc}
\usepackage{eucal,mathrsfs,dsfont}%gives nicer set-names. Use \mathds instead of \mathds
\usepackage{color}%cyan,red;magenta,green;yellow,blue
\usepackage{enumerate,paralist}

\newcommand{\jian}{\color{blue}}%other colors: magenta, green, yellow ...
\newcommand{\rene}{\color{red}}

\definecolor{mno}{rgb}{0.5,0.1,0.5}

\newcommand{\R}{\mathds R}
\newcommand{\N}{\mathds N}

\newcommand{\I}{\mathds 1}

\newcommand{\lip}{\mathrm{Lip}}
\newcommand{\loc}{\mathrm{loc}}
\newcommand{\supp}{\mathrm{supp}}
\newcommand{\diag}{\Delta}

\newtheorem{theorem}{Theorem}[section]
\newtheorem{lemma}[theorem]{Lemma}
\newtheorem{proposition}[theorem]{Proposition}
\newtheorem{corollary}[theorem]{Corollary}

\theoremstyle{definition}

\newtheorem{remark}[theorem]{Remark}
\newtheorem*{acknowledgement}{Acknowledgement}

\renewcommand{\leq}{\leqslant}
\renewcommand{\geq}{\geqslant}
\renewcommand{\le}{\leqslant}
\renewcommand{\ge}{\geqslant}

\begin{document}
\allowdisplaybreaks

\title[Semi-Dirichlet forms and L\'{e}vy type operators]{\bfseries Lower bounded semi-Dirichlet forms associated with L\'{e}vy type operators}

\author{Ren\'{e} L.\ Schilling \and Jian Wang}
\thanks{\emph{R.\ Schilling:} TU Dresden, Institut f\"{u}r Mathematische Stochastik, 01062 Dresden, Germany. \texttt{rene.schilling@tu-dresden.de}}
\thanks{\emph{J.\ Wang:} School of Mathematics and Computer Science, Fujian Normal
University, 350007, Fuzhou, P.R. China.
\texttt{jianwang@fjnu.edu.cn}}

\date{}

\maketitle

\begin{abstract}
Let $k:E\times E\to [0,\infty)$ be a non-negative measurable function on some locally compact separable metric space $E$. We provide some simple conditions such that the quadratic form with jump kernel $k$ becomes a regular lower bounded (non-local, non-symmetric) semi-Dirichlet form. If $E=\R^n$ we identify the generator of the semi-Dirichlet form and its (formal) adjoint. In particular, we obtain a closed expression of the adjoint of the stable-like generator $-(-\Delta)^{\alpha(x)}$ in the sense of Bass. Our results complement a recent paper by Fukushima and Uemura \cite{FU} and establish
 the relation of these results with the symmetric principal value (SPV) approach due to Zhi-ming Ma and co-authors \cite{HMS}.

\medskip


\end{abstract}

Let $(E,d,m)$ be a locally compact separable metric measure space. The reference measure $m$ is a Radon measure with full topological support. Recently, Fukushima and Uemura \cite{FU} were able to construct a regular lower bounded semi-Dirichlet form and the corresponding jump-type Hunt process for a given jump kernel $k(x,y)$.  A key ingredient in their construction are conditions that ensure that the symmetric part of the kernel, $k_s$, dominates the totally anti-symmetric part, $k_a$, where
$$
    k_s(x,y):=\tfrac{1}{2}\big(k(x,y)+k(y,x)\big)
    \quad\text{and}\quad
    k_a(x,y):=\tfrac{1}{2}\big(k(x,y)-k(y,x)\big).
$$
For the readers' convenience let us briefly recall these assumptions, see \cite[(2.1)--(2.4), Section 2]{FU},
\begin{align}
\tag{A0}    &x\mapsto\int_{y\neq x}\big(1\wedge d(x,y)^2\big) k_s(x,y)\,m(dy)\in L^1_{\loc}(E,m)\label{cond1},\\
\tag{A1}    &C_1:=\sup_{x\in E}\int_{d(x,y)\ge 1}|k_a(x,y)|\,m(dy)<\infty,\label{A1}\\
\tag{A2}    &C_2:=\sup_{x\in E} \int_{d(x,y)<1}|k_a(x,y)|^\gamma\,m(dy)<\infty
                \quad\text{for some}\quad\gamma\in (0,1],\label{A2}\\
\tag{A3}    &C_3 := \sup_{x,y\in E, \: 0<d(x,y)\le 1} \frac{|k_a(x,y)|^{2-\gamma}}{k_s(x,y)} < \infty
            \quad\text{for $\gamma$ from \eqref{A2}}.\label{A3}
\end{align}

In this note we will simplify these conditions. If $E=\R^n$ we obtain the explicit expressions for the generator of the form and its formal adjoint in terms of  Cauchy principal value integral (PV), which is related to the symmetric Cauchy principal value integral (SPV) in the sense of Ma et al.\ \cite{HMS}. This is motivated by and improves more recent development on non-local Dirichlet forms and L\'{e}vy type operators, e.g.\ \cite{U1,FU}.

\section{Lower Bounded Semi-Dirichlet Forms}\label{section1}
Let $(E,d,m)$ be a locally compact separable metric measure space equipped with a Radon measure $m$, and $k(x,y)$ a non-negative Borel measurable function on the space $E \times E\setminus\diag$, where $\diag$ denotes the diagonal $\{(x,x):x\in E\}$ in $E\times E$. The inner product and the norm in $L^2(E,m)$ are denoted by $\langle\cdot,\cdot\rangle_{L^2}$ and $\|\cdot\|_{L^2}$, respectively. As before, denote by $k_s$ and $k_a$ the symmetric part and the totally anti-symmetric part of $k$, respectively. Let $C_c^{\lip}(E)$ be the space of Lipschitz continuous functions on $E$ with compact support. Throughout this section, we will assume \eqref{cond1}.

A (not necessarily symmetric) bilinear form $(\eta,\mathscr F)$, $\mathscr F\subset L^2(E, m)$, is a lower bounded Dirichlet form if the following conditions are satisfied: for some $\alpha>0$
\begin{enumerate}[i)]
\item $\eta(u,u)\geq -\alpha\langle u,u\rangle_{L^2}$ for all $u\in\mathscr F$;
\item $\eta(u,v) \leq c \sqrt{\eta(u,u)+\alpha\langle u,u\rangle_{L^2}}\sqrt{\eta(v,v)+\alpha\langle v,v\rangle_{L^2}}$ for all $u,v\in\mathscr F$;
\item $(\mathscr F,\,\eta(\cdot,\cdot)+\alpha\langle\cdot,\cdot\rangle_{L^2})$ is a complete subspace of $L^2(E, m)$;
\item  $u^+\wedge 1\in\mathscr F$ for all $u\in\mathscr F$ and $\eta(u^+\wedge 1, u-u^+\wedge 1) \geq 0$.
\end{enumerate}
For further details we refer to \cite[Section 1]{FU} and the references therein.

For each $n\in N$, we define the operator $L_nu$ for $u\in C_c^{\lip}(E)$ by
\begin{equation*}\label{cond3}
    L_nu(x):=\int_{\{y\in E\,:\, d(x,y)>1/n\}} \big(u(y)-u(x)\big) k(x,y)\,m(dy),\quad x\in E,
\end{equation*}
and the quadratic form $\eta_n(u,v)$ for $u,v\in  C_c^{\lip}(E)$ by
$$
    \eta_n(u,v)
    :=-\langle L_nu, v\rangle_{L^2}
    =-\int_E L_nu(x) v(x)\,m(dx).
$$
Due to \eqref{cond1}, all integrals appearing in the definition of $L_n$ and $\eta_n$ are absolutely convergent. Finally, set
\begin{align*}
    \mathscr{E}(u,v)
    &= \iint_{y\neq x} (u(x)-u(y))(v(x)-v(y)) k_s(x,y) \,m(dx)\,m(dy),\\
    \mathscr{F}^r
    &=\Big\{u\in L^2(E,m): u \text{\ \ is Borel measurable and\ \ }\mathscr{E}(u,u)<\infty\Big\}.
\end{align*}
The condition \eqref{cond1} ensures that $(\mathscr{E}, \mathscr{F}^r)$ is a symmetric Dirichlet form on $L^2(E,m)$, and $\mathscr{F}^r$ contains the space $C_c^{\lip}(E)$. As usual, $\mathscr E_1(u,u) = \mathscr E(u,u) + \|u\|_{L^2}^2$, and we write $\mathscr{F}^0$ for the $\mathscr{E}_1$-closure of $C_c^{\lip}(E)$ in $\mathscr{F}^r$. In particular, $(\mathscr{E}, \mathscr{F}^0)$ is a regular symmetric Dirichlet form on $L^2(E,m)$, cf.\ \cite[Example 1.2.4]{FkU}.

Our main result in this section is the following simple condition which guarantees that the limit of the forms $\eta_n(u,v)$, $n\to\infty$ exists, and defines a regular lower bounded semi-Dirichlet form. This generalizes and simplifies the earlier result by Fukushima and Uemura \cite[Proposition 2.1 and Theorem 2.1]{FU}.

\begin{theorem}\label{th1}
Assume that \eqref{cond1} is satisfied and that
\begin{equation}\label{cond2}
    \sup_{x\in E} \int_{\{k_s(x,y)\neq 0\}}\frac{k_a(x,y)^2}{k_s(x,y)}\,m(dy) < \infty
\end{equation}
holds. Then we have the following two statements.

\medskip\noindent
\textup{(i)}
For all $u,v\in C_c^{\lip}(E)$, the limit
$
    \eta(u,v)=\lim_{n\rightarrow\infty}\eta_n(u,v)
$
exists. The form $\eta(u,v)$ has the following integral representation
\begin{equation}\label{th11}
    \eta(u,v)
    =\frac{1}{2}\,\mathscr{E}(u,v)+ \iint_{y\neq x}(u(x)-u(y))\,v(y) k_a(x,y) \,m(dx)\,m(dy),
\end{equation}
where the integral on the right hand side of \eqref{th11} is absolutely convergent.

\medskip\noindent
\textup{(ii)}
The form $\eta$ extends from $C_c^{\lip}(E)\times C_c^{\lip}(E)$ to $\mathscr{F}^0\times \mathscr{F}^0$ such that the pair $(\eta, \mathscr{F}^0)$ is a regular lower bounded semi-Dirichlet form on $L^2(E,m)$.
\end{theorem}

Let us briefly show that the conditions imposed by Fukushima and Uemura are more restrictive than \eqref{cond2}. Indeed, if \eqref{A1}--\eqref{A3} hold, then we find for $x\in E$,
\begin{align*}
&\int_{k_s(x,y)\neq 0} \frac{k_a(x,y)^2}{k_s(x,y)}\,m(dy)\\
&\le\int_{\begin{subarray}{c}d(x,y)\le 1,\\ {k_s(x,y)\neq 0}\end{subarray}}
    \frac{|k_a(x,y)|^{2-\gamma}}{k_s(x,y)}\,|k_a(x,y)|^\gamma\,m(dy)
    +\int_{d(x,y)> 1}k_s(x,y)\,m(dy)\\
&\le\left\{\sup_{\begin{subarray}{c}d(x,y)\le 1,\\ {k_s(x,y)\neq 0}\end{subarray}}
    \frac{|k_a(x,y)|^{2-\gamma}}{k_s(x,y)}\right\} \int_{d(x,y)\le 1} |k_a(x,y)|^\gamma\, m(dy)
    + \int_{d(x,y)> 1}k_s(x,y)\,m(dy)\\
&\le C_2C_3+C_1.
\end{align*}
In the first inequality we have used that $|k_a(x,y)|\le k_s(x,y)$.

\begin{proof}[Sketch of the proof of Theorem \ref{th1}]
From the definition of $\eta_n$ we find for all $u,v\in C_c^{\lip}(E)$ that
$$
    \eta_n(u,v)
    =\frac{1}{2}\,\mathscr{E}_n(u,v)+\iint_{d(x,y)\ge 1/n} (u(x)-u(y))v(y)k_a(x,y)\,m(dx)\,m(dy),
$$
where
$$
    \mathscr{E}_n(u,v)=\iint_{d(x,y)\ge 1/n}(u(x)-u(y))(v(x)-v(y)) k_s(x,y) \,m(dx)\,m(dy).
$$
Because of \eqref{cond1}, $\mathscr{E}_n(u,v)$ converges to $\mathscr{E}(u,v)$ as $n\rightarrow\infty$.

To see the convergence of the non-symmetric part, we set for $x\in E$
$$
    h(x):=\int_{k_s(x,y)\neq 0} \frac{k_a(x,y)^2}{k_s(x,y)}\,m(dy).
$$
An application of the Cauchy-Schwarz inequality and \eqref{cond2} show
\begin{align*}
&\iint\limits_{d(x,y)\ge 1/n} |u(x)-u(y)||v(y)||k_a(x,y)|\,m(dx)\,m(dy)\\
&=\iint\limits_{\begin{subarray}{c}d(x,y)\ge 1/n,\\{k_s(x,y)\neq 0}\end{subarray}} |u(x)-u(y)| k_s(x,y)^{1/2}\cdot |v(y)| |k_a(x,y)| k_s(x,y)^{-1/2}\,m(dx)\,m(dy)\\
&\le \Bigg[\;\;\iint\limits_{d(x,y)\ge 1/n} (u(x)-u(y))^2 k_s(x,y)\,m(dx)\,m(dy)\Bigg]^{\frac 12} \times\\
&\qquad\qquad\times\Bigg[\;\;\iint\limits_{\begin{subarray}{c}d(x,y)\ge 1/n,\\{k_s(x,y)\neq 0}\end{subarray}} v(y)^2\frac{k_a(x,y)^2}{k_s(x,y)}\,m(dx)\,m(dy)\Bigg]^{\frac 12}\\
&\le  \Big[\mathscr{E}_n(u,u)\Big]^{\frac 12} \bigg[\int v^2(y) h(y)\,m(dy)\bigg]^{\frac 12}\\
&\le \Big[\mathscr{E}(u,u)\Big]^{\frac 12}  \|h\|_\infty^{
\frac 12}  \|v\|_{L^2}.
\end{align*}
This shows that the expression
$$
    \iint_{d(x,y)\ge 1/n} (u(x)-u(y))v(y)k_a(x,y)\,m(dx)\,m(dy)
$$
converges absolutely as $n\rightarrow\infty$ and (i) follows. In order to see (ii), we use (i) and the argument used in the proof of \cite[Theorem 2.1]{FU}.
\end{proof}

The semi-Dirichlet form $(\eta, \mathscr{F}^0)$ given by \eqref{th11} is a coercive closed form in the sense of Ma-R\"{o}ckner, cf.\ \cite[Chapter I, Definition 2.4, page 16]{MR}. Then, by \cite[Chapter I, Proposition 2.16, page 23]{MR}, $(L,D(L))$ is the (pre-)generator of the form $(\eta, \mathscr{F}^0)$, where
$D(L)=\{u\in \mathscr{F}^0\,|\, v\mapsto \eta(u,v)\text{\ is continuous with respect to\ }\|\cdot\|_{L^2}\text{\ on\  }\mathscr{F}^0\}.$  According to \cite[Chapter I, Theorem 2.15, page 22]{MR}, the generator $(L,D(L))$ is a
linear operator mapping $D(L)$ into
$L^2(E,m)$ such that
\begin{equation}\label{pregenerator}
    \eta(u,v)=-\langle Lu, v\rangle_{L^2}, \qquad u\in D(L),\,\, v\in \mathscr{F}^0.
\end{equation} However, for the semi-Dirichlet form $\eta(u,v)$ given by \eqref{th11}, it is in general difficult to find a closed expression for the generator $(L, D(L))$.

Note that any semi-Dirichlet form can be uniquely decomposed into three terms, which involve the integral in the sense of the symmetric Cauchy principal value (\text{SPV}), cf.\ see \cite[Definition 2.5, Theorems 2.6 and 4.1]{HMS}.
Motivated by this fact, we can obtain some information on $(L,D(L))$ if we assume that
\begin{equation}\label{cond4}
    \int_{0<d(x,y)\le 1}d(x,y)|k_a(x,y)|\,m(dy)<\infty,\qquad x\in E.
\end{equation}

\begin{proposition}\label{levy}
Assume that \eqref{cond1}, \eqref{cond2} and \eqref{cond4} hold. Then,
$$
    \eta(u,v)=-\langle Bu,v\rangle_{L^2},\quad
    u\in C^*(E), \; v\in C_c^{\lip}(E),
$$
where
$$
    C^*(E)
    :=\Big\{u\in C_c^{\lip}(E)\,:\, Bu\text{\ exists and belongs to\ } L^2(E,m)\Big\}
$$
and
\begin{equation*}
    Bu(x): = \mathrm{PV}\!\int_{y\neq x}\big(u(y)-u(x)\big)k_s(x,y)\,m(dy)
    +\int_{y\neq x}\big(u(y)-u(x)\big)k_a(x,y)\,m(dy);
\end{equation*}
$\mathrm{PV}\!\int\cdots dm$ indicates the Cauchy principal value, i.e.\ for any $x\in E$, the limit
$$
    \lim_{j\to\infty} \int_{\{y\in E\,:\, d(x,y)\ge 1/j\}} \big(u(y)-u(x)\big)k_s(x,y)\,m(dy).
$$
\end{proposition}

%\begin{proof} Because of \eqref{cond4}, the operator $L$ is well defined.
%Under \eqref{cond1} and \eqref{cond2}, Theorem \ref{th1} shows that
%$$\eta(u,v)=\frac{1}{2}\mathscr{E}(u,v)+ \iint_{y\neq
%x}(u(x)-u(y))v(y) k_a(x,y) \,m(dx)\,m(dy).$$ By changing $x$ and $y$,
%$$\eta(u,v)=\frac{1}{2}\mathscr{E}(u,v)-\iint_{y\neq
%x}(u(x)-u(y))v(x) k_a(x,y) \,m(dx)\,m(dy).$$ Then the required assertion follows from the dominated convergence theorem.
%\end{proof}

%\fbox{Jian, you probably mean}
\begin{proof}
    The condition \eqref{cond4} ensures that the operator $B$ is well defined for all $x$. According to the proof of Theorem \ref{th1}, \eqref{cond1} and \eqref{cond2} imply that for any $u\in C^*(E)$ and $ v\in C_c^{\lip}(E)$, the form $\langle Bu, v\rangle_{L^2}$ is also well defined and finite.

    On the other hand, under \eqref{cond1} and \eqref{cond2}, Theorem \ref{th1} shows that
$$
    \eta(u,v)=\frac{1}{2}\lim_{j\to\infty}\mathscr{E}_j(u,v)+ \iint_{y\neq x}(u(x)-u(y))v(y) k_a(x,y) \,m(dx)\,m(dy),
$$
    with
$$
    \mathscr E_j(u,v) = \iint_{d(x,y)\ge 1/j} (u(x)-u(y))(v(x)-v(y))k_s(x,y)\,m(dx)\,m(dy).
$$
    Since $k_s(x,y)=k_s(y,x)$,
\begin{align*}
    \mathscr E_j(u,v)
    &= \frac 12\iint_{d(x,y)\ge 1/j} (u(x)-u(y))v(x)k_s(x,y)\,m(dx)\,m(dy) \\
    &\qquad- \frac 12\iint_{d(x,y)\ge 1/j} (u(x)-u(y))v(y)k_s(x,y)\,m(dx)\,m(dy)\\
    &= \frac 12\iint_{d(x,y)\ge 1/j} (u(y)-u(x))v(y)k_s(y,x)\,m(dx)\,m(dy)\\
    &\qquad- \frac 12\iint_{d(x,y)\ge 1/j} (u(x)-u(y))v(y)k_s(x,y)\,m(dx)\,m(dy)\\
    &= \iint_{d(x,y)\ge 1/j} (u(y)-u(x))v(y)k_s(x,y)\,m(dx)\,m(dy),
\end{align*}
and the claim follows by the dominated convergence theorem.
\end{proof}

To get an explicit expression for the generator associated with the semi-Dirichlet form $\eta(u,v)$, we now need to characterize the domain $C^*(E)$.
\begin{theorem} \label{theorem3} Assume that $E=\R^n$ is equipped with the Euclidean metric $d(x,y)=|x-y|$ and Lebesgue measure $m(dx)=dx$. Suppose that
\begin{equation}\label{l2cond} x\mapsto\int_{y\neq x}\big(1\wedge |y-x|^2\big) k_s(x,y)\,dy\in L^2_{\loc}(dx),\end{equation}
\begin{equation}\label{12cond1} x\mapsto\int_{0<|z|\le 1}|z||k_s(x,x+z)-k_s(x,x-z)|\,dz \in L^2_{\loc}(dx)\end{equation}  and
\begin{equation}\label{l2cond2}x\mapsto\int_{\{y\in\R^n\,:\, |y-x|\ge 1\}} k_s(x,y)\,dy\in L^2(dx)\cup L^\infty (dx).\end{equation}
Let $(L, D(L))$ be a generator associated through \eqref{pregenerator} with the semi-Dirichlet form $\eta(u,v)$ given by \eqref{th11}. If \eqref{cond2} and \eqref{cond4} hold, then the set of twice differentiable functions with compact support is in the domain of $L$, i.e.\ $C_c^2(\R^n)\subset D(L)$, and on $C_c^2(\R^n)$ the operator $L$ is of the following form
 \begin{equation}\label{cond3}\aligned
    Lu(x)    &=\int_{z\neq 0}\big(u(x+z)-u(x)-\nabla u(x) \cdot z\I_{\{|z|\le 1\}}\big)k_s(x,x+z)\,dz\\
    &\quad    +\frac{1}{2}\nabla u(x)\cdot \int_{0<|z|\le 1} z\big(k_s(x,x+z)-k_s(x,x-z)\big) \,dz\\
    &\quad+ \int_{y\neq x}\big(u(y)-u(x)\big)k_a(x,y)\,dy.\endaligned
\end{equation}
\end{theorem}
\begin{proof}\emph{Step 1:} Note that the set $C_c^{\lip}(\R^n)$ is dense in both $\mathscr{F}^0$ and $D(L)$ with respect to $\|\cdot\|_{L^2}$. According to Proposition \ref{levy}, under \eqref{cond2}, \eqref{cond4} and \eqref{l2cond}, the operator $L$ has the following form
$$
 Lu(x)= \mathrm{PV}\!\int_{y\neq x} \big(u(y)-u(x)\big)k_s(x,y)\,dy +\int_{y\neq x}\big(u(y)-u(x)\big)k_a(x,y)\,dy
$$ on the set $C^*(\R^n)\cap D(L)$. Therefore, to get the required assertion it is sufficient to check that $C_c^2(\R^n) \subset C^*(\R^n)\cap D(L)$, and that the principal value integral is the same as \eqref{cond3}.

   First, since for any $\varepsilon\ge0$, $$
    \int_{\varepsilon\le |z|\le 1} z \,(k_s (x,x+z)+k_s(x,x-z))\, dz=0,
$$ it holds for any $u\in C_c^2(\R^n)$ that
\begin{equation}\label{rrr}\begin{aligned}
&\mathrm{PV}\!\int_{y\neq x}\big(u(y)-u(x)\big)k_s(x,y)\,dy\\
&= \mathrm{PV}\!\int_{z\neq 0}\big(u(x+z)-u(x)\big)k_s(x,x+z)\,dz\\
&= \lim_{\varepsilon\to0} \int_{|z|\ge \epsilon}\big(u(x+z)-u(x)\big)k_s(x,x+z)\,dz\\
&= \lim_{\varepsilon\to0} \bigg[\int_{|z|\ge \epsilon}\big(u(x+z)-u(x)\big)k_s(x,x+z)\,dz\\
&\qquad\qquad-\frac{1}{2} \int_{\varepsilon\le |z|\le1} z \,(k_s (x,x+z)+k_s(x,x-z))\, dz \cdot\nabla u(x)\bigg]\\
&= \lim_{\varepsilon\to0} \bigg[\int_{|z|\ge \epsilon}\big(u(x+z)-u(x)-\nabla u(x)\cdot z\I_{\{|z|\le 1\}}\big)k_s(x,x+z)\,dz\\
&\qquad\qquad+ \frac{1}{2}\int_{\varepsilon\le |z|\le1}z \,(k_s (x,x+z)-k_s(x,x-z))\, dz\cdot\nabla u(x)\bigg]\\
&= \int_{z\neq 0}\big(u(x+z)-u(x)-\nabla u(x) \cdot z\I_{\{|z|\le 1\}}\big)k_s(x,x+z)\,dz\\
    &\qquad\qquad    +\frac{1}{2}\nabla u(x)\cdot \int_{0<|z|\le 1} z\big(k_s(x,x+z)-k_s(x,x-z)\big) \,dz,
\end{aligned}\end{equation}
where the last equality follows from \eqref{l2cond}, \eqref{12cond1} and the dominated convergence theorem. Hence, \eqref{rrr} immediately yields that for any $u\in C_c^2(\R^n)$, $\mathrm{PV}\!\int_{y\neq x} \big(u(y)-u(x)\big)k_s(x,y)\,dy$ exists, and it also gives us the explicit expression \eqref{cond3} for $L$ on $C_c^2(\R^n)$.

\medskip\noindent
\emph{Step 2:}  Because of Step 1 and \eqref{pregenerator}, to complete the proof we only need to verify that the operator $L$ maps $C_c^2(\R^n)$ into $L^2(\R^n)$.

Let
\begin{align*}I_x   & = \mathrm{PV}\!\int_{y\neq x}\big(u(y)-u(x)\big)k_s(x,y)\,dy\\
&=\int_{z\neq 0}\big(u(x+z)-u(x)-\nabla u(x) \cdot z\I_{\{|z|\le 1\}}\big)k_s(x,x+z)\,dz\\
    &\quad    +\frac{1}{2}\nabla u(x)\cdot \int_{0<|z|\le 1} z\big(k_s(x,x+z)-k_s(x,x-z)\big) \,dz\end{align*} and
$I\!\!I_x=   \int_{y\neq x}\big(u(y)-u(x)\big)k_a(x,y)\,dy.$ By the Cauchy-Schwarz inequality,
 \begin{align*}
   \|I\!\!I_x\|_{L^2(dx)}^2&=\int\bigg( \int_{y\neq x}\big(u(y)-u(x)\big)k_a(x,y)\,dy\bigg)^2\,dx\\
   &=\int\bigg( \int_{k_s(x,y)\neq 0}\big(u(y)-u(x)\big)\sqrt{k_s(x,y)}\frac{k_a(x,y)}{\sqrt{k_s(x,y)}}\,dy\bigg)^2\,dx\\
   &\le \int \bigg(\int \big(u(y)-u(x)\big)^2 k_s(x,y)\, dy\bigg)\bigg( \int_{k_s(x,y)\neq 0} \frac{k^2_a(x,y)}{k_s(x,y)}\,dy\bigg) \,dx\\
   &\le \bigg[\sup_{x\in\R^n} \int_{k_s(x,y)\neq 0} \frac{k^2_a(x,y)}{k_s(x,y)}\,dy\bigg] \, \mathscr{E}(u,u)<\infty,
\end{align*} where $\frac{1}{2}\mathscr{E}(u,u)$ is the symmetric part of $\eta(u,u)$ given by \eqref{th11}. On the other hand, for any $r>0$, it holds that
$$\|I_x\|_{L^2(dx)}^2\le \|\I_{B_{2r}(0)}(x)I_x\|_{L^2(dx)}^2+\|\I_{B_{2r}^c(0)}(x) I_x\|_{L^2(dx)}^2.$$
First, \begin{align*}
   \|\I_{B_{2r}(0)}&(x)I_x\|_{L^2(dx)}^2=\int_{|x|\le 2r}\bigg(\mathrm{PV}\!\int \big(u(y)-u(x)\big) k_s(x,y)\, dy\bigg)^2\,dx\\
   &\le 2\,\big (\|u\|_\infty \vee \|\nabla^2 u\|_\infty\big) \int_{|x|\le 2r}\bigg(\int \Big(1\wedge |x-y|^2\Big) k_s(x,y)\,dy\bigg)^2\, dx\\
   &\quad +\frac{1}{2}\|\nabla u\|_\infty\int_{|x|\le 2r}\bigg(\int_{0<|z|\le 1}|z| \,|k_s (x,x+z)-k_s(x,x-z)|\, dz\bigg)^2\,dx\\
   &<\infty.
\end{align*} In the first inequality we have used \eqref{rrr}, and the last inequality follows from \eqref{l2cond} and \eqref{12cond1}. Pick $r>1$ large enough such that $\supp\, u\subset B_r(0)$. We get that
\begin{align*}
   \|\I_{B_{2r}^c(0)}(x)I_x\|_{L^2(dx)}^2&=\int_{|x|\ge 2r}\bigg(\mathrm{PV}\!\int \big(u(y)-u(x)\big) k_s(x,y)\, dy\bigg)^2\,dx\\
   &= \int_{|x|\ge 2r} \bigg(\int u(y)k_s(x,y)\,dy\bigg)^2\,dx\\
   &=  \int_{|x|\ge 2r} \bigg(\int_{ |y|\le r} u(y)k_s(x,y)\,dy\bigg)^2\,dx\\
   &\le\|u\|_\infty^2 \int \bigg(\int_{|x-y|>r} \I_{B_r(0)}(y) k_s(x,y)\,dy\bigg)^2\,dx,
\end{align*} where the second equality follows again from \eqref{rrr}.

If $\displaystyle x\mapsto\int_{|y-x|\ge 1} k_s(x,y)\,dy\in L^2(dx)$, then we have $\|\I_{B_{2r}^c(0)}(x)I_x\|_{L^2(dx)}^2<\infty$.

If $\displaystyle x\mapsto\int_{|y-x|\ge 1} k_s(x,y)\,dy\in L^\infty(dx)$, then, by the Cauchy-Schwarz inequality,
\begin{align*}&\|u\|_\infty^2 \int \bigg(\int_{|x-y|>r} \I_{B_r(0)}(y) k_s(x,y)\,dy\bigg)^2\,dx\\
  &\le \|u\|_\infty^2 \int \bigg(\int_{|x-y|>r} \I_{B_r(0)}(y) k_s(x,y)\,dy\bigg)\bigg( \int_{|x-y|>r} k_s(x,y)\,dy\bigg)\, dx\\
  &\le \|u\|_\infty ^2 \bigg[ \sup_{x\in \R^n} \int_{|x-y|>r} k_s(x,y)\,dy \bigg]\iint_{|x-y|>r} \I_{B_r(0)}(y) k_s(x,y)\,dy\,dx\\
  &=\|u\|_\infty ^2 \bigg[\sup_{x\in \R^n} \int_{|x-y|>r} k_s(x,y)\,dy\bigg] \int \I_{B_r(0)}(x) \int_{|x-y|>r}  k_s(x,y)\,dy\,dx
  \;<\;\infty,
\end{align*}
   where in the equality above we have used the symmetry of $k_s(x,y)\,dy\,dx,$ and the last inequality follows from \eqref{cond1}. This also gives us that $\|\I_{B_{2r}^c(0)}(x)I_x\|_{L^2(dx)}^2<\infty$.
   The required assertion follows from all the conclusions above.
    \end{proof}

\section{The Adjoint of a L\'{e}vy Type Operator on $\R^n$}
Assume that $E=\R^n$ is equipped with the Euclidean metric $d(x,y)=|x-y|$ and Lebesgue measure $m(dx)=dx$ as reference measure. If $k$ is symmetric, then $k_a(x,y)=0$ and Proposition \ref{levy} is identical with \cite[Theorem 2.2]{SU}. As shown by Theorem \ref{theorem3}, in this case the Cauchy principal value integral in the representation of $L$ can be rewritten as an absolutely convergent integral if we introduce a regularizing term in the integrand, and so the expression of $L$ becomes \eqref{cond3}, which is a kind of symmetric L\'{e}vy type operator, see \cite{WJ}.
This observation enables us to consider the (formal) adjoint of general (not necessarily symmetric) L\'{e}vy type operators.

Let $C_c^\infty(\R^n)$ be the space of smooth functions with compact support on $\R^n$. For $f\in C_c^\infty(\R^n)$, define the following L\'{e}vy type operator
\begin{equation}\label{du1}\begin{aligned}
    Lf(x)
    &=\int_{z\neq0}\Bigl(f(x+z)-f(x)-\nabla f(x)\cdot z{\I}_{\{|z|\le1\}}\Bigr)j(x,x+z)\,dz\\
    &\quad +\frac{1}{2}\int_{0<|z|\le 1}z\,\big(j(x,x+z)-j(x,x-z)\big)\,dz\cdot\nabla f(x),
\end{aligned}\end{equation}
where $\int_{z\neq 0}\big(1\wedge |z|^2\big)j(x,x+z)\,dz<\infty$ and
$$
    \int_{0<|z|\le 1}|z|\big|j(x,x+z)-j(x,x-z)\big|\,dz<\infty
$$
for all $x\in \R^n$.

 We will now present an explicit expression of the (formal) adjoint of the operator $L$. To state our result, we need a few assumptions. As before we write $j_s$ and $j_a$ for the symmetric and antisymmetric parts of $j$, i.e.\
$$
    j_s(x,y) := \tfrac 12\big(j(x,y)+j(y,x)\big)
    \quad\text{and}\quad
    j_a(x,y) := \tfrac 12\big(j(x,y)-j(y,x)\big).
$$ For $x,z\in\R^n$, we denote by
$$j^*(x,z):= \big|j(x,x+z)-j(x,x-z)\big|+\big|j(x+z,x)-j(x-z,x)\big|.$$

\begin{align}
\tag{H1}\label{H1}
    &x\mapsto\int \big(1\wedge (y-x)^2\big) j_s(x,y)\,dy \in L^2_{\loc}(dx);\\
     \tag{H2}\label{H2}
    &x\mapsto\int_{\{y\in\R^n\,:\, |y-x|\ge 1\}} j_s(x,y)\,dy\in L^2(dx)\cup L^\infty (dx);\\
\tag{H3}\label{H3}
    &x\mapsto\int_{0<|z|\le 1}|z|j^*(x,z)\,dz \in L^2_{\loc}(dx);\\
\tag{H4}\label{H4}
    &\sup_{x\in\R^n} \int_{j_s(x,y)\neq 0} \frac{j_a(x,y)^2}{j_s(x,y)}\,dy < \infty;\\
\tag{H5}\label{H5}
    &\sup_{x\in K} \sup_{\epsilon > 0} \left|\int_{|y-x|\ge \varepsilon} j_a(x,y)\,dy\right| < \infty
    \quad\text{for every compact set}\quad K\subset\R^n.
\end{align}
Note that \eqref{H1} is just \eqref{l2cond}, which implies \eqref{cond1}. \eqref{H2} is \eqref{l2cond2}, and \eqref{H4} is the same as \eqref{cond2}. \eqref{H3} implies
\begin{equation}\label{impli2} x\mapsto\int_{0<|z|\le 1}|z|\big|j_s(x,x+z)-j_s(x,x-z)\big|\,dz \in L^2_{\loc}(dx),\end{equation} which is just \eqref{12cond1}.  Although \eqref{H5} has no direct counterpart in Section \ref{section1}, it is satisfied when \eqref{A2} and \eqref{A3} hold with $\gamma=1$ (as it is assumed in \cite[Section 3]{FU}). As in the proof of Theorem \ref{theorem3}, we can easily obtain that, under \eqref{H1}--\eqref{H3}, the operator $L$ given by \eqref{du1} maps $C_c^\infty(\R^n)$ into $L^2(\R^n)$.

For any $f\in C_c^\infty(\R^n)$, define \begin{equation}\label{du12300}\begin{aligned}
    \Lambda f(x):
    &= \int_{z\neq0}\big(f(x+z)-f(x)-\nabla f(x)\cdot z{\I}_{\{|z|\le1\}}\big)\,j(x+z,x)\,dz\\
    &\quad +\frac{1}{2}\int_{0<|z|\le 1}z\,\big(j(x+z,x)-j(x-z,x)\big)\,dz\cdot\nabla f(x).
\end{aligned}\end{equation}

\begin{theorem}\label{dual}
Assume that \eqref{H1}--\eqref{H5} hold.

\textup{(i)} The operators $(L, C_c^\infty(\R^n))$ and $(\Lambda, C_c^\infty(\R^n))$ given by \eqref{du1} and \eqref{du12300}, respectively, are pre-generators corresponding to the semi-Dirichlet forms $(\eta_L, \mathscr{F}^0\times \mathscr{F}^0)$ and $(\eta_\Lambda, \mathscr{F}^0\times \mathscr{F}^0)$, generated by the kernels $j(x,y)$ and $j(y,x)$  on $L^2(\R^n)$ in the sense of \eqref{pregenerator} with $D(L)=D(\Lambda)=C_c^\infty(\R^n)$.

\textup{(ii)} Let $(L^*, C_c^\infty(\R^n))$ be the dual operator for $L$. Then, for any $f\in C_c^\infty(\R^n)$
\begin{equation}\label{du123}\begin{aligned}
    L^*f(x)
    &= \Lambda f(x)
    + \kappa(x)f(x),
\end{aligned}\end{equation}
where $\kappa(x)$ is a measurable function on $\R^n$ such that $\kappa(x)\,dx$ is the vague limit of the sequence of (signed) measures $\big\{\big(-2\int_{|x-y|>1/m} j_a(x,y)\,dy\big)\,dx\big\}_{m\in\N}$.
\end{theorem}

By \eqref{H1}, \eqref{impli2} and \eqref{H5}, the operator $L^*$ given by \eqref{du123} is well defined on $C_c^\infty(\R^n)$. From the proof of Theorem \ref{theorem3}, we know that, under \eqref{H1}--\eqref{H3} and \eqref{H5}, the operator $L^*$ maps $C_c^\infty(\R^n)$ into $L^2(\R^n).$ In particular, if $j$ is symmetric, i.e.\ $j(x,y)=j(y,x)$ for all $x,y\in\R^n$, Theorem \ref{dual} shows that $L=\Lambda=L^*$ on $C_c^\infty(\R^n)$; that is, $L$ defined by \eqref{du1} is a symmetric L\'{e}vy type operator. For further details we refer to \cite[Theorem 2.2]{SU}
and \cite[Theorem 1.2]{WJ}.

\begin{proof} [Proof of Theorem \ref{dual}]
The proof is divided into four steps. Throughout the proof we fix some $f,g\in C_c^\infty(\R^n)$. Note that the conditions on $j(x,y)$ in assumptions \eqref{H1}-\eqref{H5} are symmetric with respect to $x$, $y\in\R^n$. Therefore it is enough to prove part (i) for the operator $L$.

\medskip\noindent
\emph{Step 1:} For any $\varepsilon>0$, we define
\begin{align*}
    L_\varepsilon f(x)
    &:= \int_{|z|\ge \varepsilon}\Big(f(x+z)-f(x)-\nabla f(x)\cdot z{\I}_{\{|z|\le1\}}\Big)j(x,x+z)\, dz\\
    &\quad + \frac{1}{2}\int_{\varepsilon\le|z|\le1}z\,\big(j(x,x+z)-j(x,x-z)\big)\,dz\cdot\nabla f(x)\\
    &= \int_{|z|\ge \varepsilon}\Big(f(x+z)-f(x)\Big)j(x,x+z)\, dz\\
    &\quad -\frac{1}{2}\int_{\varepsilon\le|z|\le1}z\,\big(j(x,x+z)+j(x,x-z)\big)\,dz\cdot\nabla f(x).
\end{align*}
Since for every $x\in\R^n$,
$$
    \int_{\varepsilon\le|z|\le1}z\,\big(j(x,x+z)+j(x,x-z)\big)\,dz=0,
$$
we get
$$
    L_\varepsilon f(x)=\int_{|z|\ge \varepsilon} \Big(f(x+z)-f(x)\Big)j(x,x+z)\,dz;
$$
due to \eqref{H1}, $L_\varepsilon f \in L^2(\R^n)$. Recall that $\langle f, g\rangle_{L^2} = \int f(x)g(x)\,dx$ denotes the inner product in $L^2(\R^n)$. We have
\begin{equation}\label{pregenerator1}\begin{aligned}
    \langle L_\varepsilon f, g\rangle_{L^2}
    &= \int g(x) \int_{|z|\ge \varepsilon}\big(f(x+z)-f(x)\big)j(x,x+z)\,dz\,dx\\
    &= \int g(x) \int_{|y-x|\ge \varepsilon} \big(f(y)-f(x)\big)j(x,y)\,dy\,dx.
\end{aligned}\end{equation}

Note that under \eqref{H1}--\eqref{H3}, for any $f\in C_c^\infty(\R^n)$, \begin{equation}\label{pregenerator2}\lim_{\varepsilon\to 0}\|L_\varepsilon f-L f\|_{L^2}=0.\end{equation}
This along with \eqref{pregenerator1}, \eqref{H4} and Theorem \ref{th1} yields (i) for the operator $L$.

\medskip\noindent
\emph{Step 2:} We will now prove (ii). We begin with showing that the limit
\begin{equation}\label{pdu2}
    A(f,g):=-\frac{1}{2}\lim_{\varepsilon\rightarrow0}\bigg[\langle L_\varepsilon f,g\rangle_{L^2}+\langle f,L_\varepsilon g\rangle_{L^2}\bigg]
\end{equation}
exists. Observe that
\begin{align*}
    \langle L_\varepsilon f&,g\rangle_{L^2}+\langle f,L_\varepsilon g\rangle_{L^2}\\
    &=\bigg[\int g(x)\int_{|y-x|\ge \varepsilon} \big(f(y)-f(x)\big)j(x,y)\,dy\,dx\\
    &\qquad+\int f(x)\int_{|y-x|\ge \varepsilon}\big(g(y)-g(x)\big)j(x,y)\,dy\,dx\bigg]\\
    &=\iint_{|y-x|\ge \varepsilon}\Big(\big(f(y)-f(x)\big)g(x)+\big(g(y)-g(x)\big)f(x)\Big)j(x,y)\,dx\,dy\\
    &=\iint_{|y-x|\ge\varepsilon} \Big(\!\big(f(y)-f(x)\big)g(x)+\big(g(y)-g(x)\big)f(x)\Big)j_s(x,y)\,dx\,dy\\
    &\qquad+\iint_{|y-x|\ge\varepsilon}\Big(\big(f(y)-f(x)\big)g(x)+\big(g(y)-g(x)\big)f(x)\Big)j_a(x,y)\,dx\,dy\\
    &=:I_1^\varepsilon+I_2^\varepsilon.
\end{align*}

If we change $x$ and $y$ in the expression of $I_1^\varepsilon$, we get
$$
    I_1^\varepsilon
    = -\iint_{|y-x|\ge\varepsilon}\Big(\big(f(y)-f(x)\big)g(y)+\big(g(y)-g(x)\big)f(y)\Big)j_s(x,y)\,dx\,dy,
$$
which, if added to the original expression for $I_1^\varepsilon$, yields
that
\begin{equation*}
I_1^\varepsilon
=-\iint_{|y-x|\ge \varepsilon} \big(f(y)-f(x)\big)\big(g(y)-g(x)\big)j_s(x,y)\,dx\,dy.
\end{equation*}
Because of \eqref{H1} we find
\begin{equation}\label{pdu3}
    I_1: = \lim_{\varepsilon\to 0}I_1^\varepsilon
    =-\iint_{x\neq y} \big(f(y)-f(x)\big)\big(g(y)-g(x)\big)j_s(x,y)\,dx\,dy.
\end{equation}

On the other hand, we see as in the proof of Theorem \ref{th1}, that under \eqref{H1} and \eqref{H4}, the limit $I_2: = \lim_{\varepsilon\to 0} I_2^\epsilon$ exists and
\begin{equation}\label{pdu4}
    I_2=\iint\Big(\big(f(y)-f(x)\big)g(x)+\big(g(y)-g(x)\big)f(x)\Big)j_a(x,y)\,dx\,dy
\end{equation}
with an absolutely convergent integral.

\medskip\noindent
\emph{Step 3:}
According to \cite[Theorem 2.2]{SU}, the assumptions \eqref{H1} and \eqref{H3} imply that
\begin{equation}\label{pdu5}
    -\frac 12\,I_1=-\langle\widetilde{L} f,g\rangle_{L^2},
\end{equation}
where
\begin{equation*}\begin{aligned}
    \widetilde{L}f(x)
    &:= \int_{z\neq 0}\Big(f(x+z)-f(x)-\nabla f(x)\cdot z{\I}_{\{|z|\le1\}}\Big) j_s(x,x+z)\,dz\\
    &\quad +\frac{1}{2}\int_{0<|z|\le 1}z\,\big(j_s(x,x+z)-j_s(x,x-z)\big)\,dz\cdot\nabla f(x).
\end{aligned}\end{equation*}
With the same reasoning as above, the proof of Theorem \ref{theorem3} shows that, under \eqref{H1}---\eqref{H3}, the operator $ \widetilde{L}$ maps $C_c^\infty(\R^n)$ into $L^2(\R^n)$.

If we change $x$ and $y$ in the expression of $I_2$, we get because of the antisymmetry of $j_a$
$$
    I_2
    =\iint_{x\neq y} \Big(\big(f(y)-f(x)\big)g(y)+\big(g(y)-g(x)\big)f(y)\Big)j_a(x,y)\,dx\,dy
$$
and if we add this to the original expression for $I_2$, we see
$$
    I_2
    =\iint_{x\neq y} \Big(f(y)g(y)-f(x)g(x)\Big)j_a(x,y)\,dx\,dy.
$$

In the same way we find that
\begin{align*}
    I_2^\varepsilon
    &= \iint_{|y-x|\ge \varepsilon}\Big(f(y)g(y)-f(x)g(x)\Big)j_a(x,y)\,dx\,dy\\
    &= \iint_{|y-x|\ge \varepsilon}f(y)g(y)\,j_a(x,y)\,dx\,dy
        -\iint_{|y-x|\ge \varepsilon}f(x)g(x)\,j_a(x,y)\,dx\,dy\\
    &= -\iint_{|y-x|\ge \varepsilon}f(x)g(x)\,j_a(x,y)\,dx\,dy
        -\iint_{|y-x|\ge \varepsilon}f(x)g(x)\,j_a(x,y)\,dx\,dy\\
    &= -2\int f(x)g(x) \bigg[\int_{|y-x|\ge \varepsilon} j_a(x,y)\,dy\bigg] dx.
\end{align*}
Since $f,g\in C_c^\infty(\R^n)$ are arbitrary and since the limit $\lim_{\varepsilon\to 0} I_2^\varepsilon = I_2$ exists, we will see from Lemma \ref{lemmabelow} below that the vague limit of the sequence of (signed) measures $\big\{\big(-2\int_{|x-y|>\varepsilon} j_a(x,y)\,dy\big)\,dx\big\}_{\varepsilon>0}$, $\varepsilon\to0$ exists and possesses a density function $\kappa(x)$ with respect to Lebesgue measure.

Thus, by \eqref{H5} again, for all $f,g\in C_c^\infty(\R^n)$,
\begin{equation}\label{pdu6}
    I_2 = \int f(x)g(x)  \, \kappa(x)\,dx= \langle\kappa f, g\rangle_{L^2} .
\end{equation}

\medskip\noindent
\emph{Step 4:} Since $C_c^\infty(\R^n)$ is dense in $L^2(\R^n)$, the formal adjoint $L^*$ of the operator $L$ satisfies that
$$
    \langle{{L}^*} f,g\rangle_{L^2}
    =\langle f, L g\rangle_{L^2},\qquad f,g\in C_c^\infty(\R^n).
$$
On the other hand, according to \eqref{pregenerator2}, for any $f,g\in C_c^\infty(\R^n)$,
\begin{equation}\label{pdu7}
    A(f,g)=-\frac{1}{2}\lim_{\varepsilon\rightarrow0}\bigg[\langle L_\varepsilon f,g\rangle_{L^2}+\langle f,L_\varepsilon g\rangle_{L^2}\bigg]=-\frac{1}{2}\bigg[\langle L f,g\rangle_{L^2}+\langle f,Lg\rangle_{L^2}\bigg].
\end{equation}

Combining \eqref{pdu2}--\eqref{pdu7}, we have
$$
    \langle {L}^* f,g\rangle_{L^2}+\langle Lf, g\rangle_{L^2}
    = \langle f,Lg\rangle_{L^2}+\langle Lf, g\rangle_{L^2}
    =\Big(2\langle{\widetilde{{L}}} f,g\rangle_{L^2}+\langle \kappa f, g\rangle_{L^2}\Big).
$$
Therefore,
$$
    L^*=(2\widetilde{L}+\kappa)-L
$$
which is what we have claimed.
\end{proof}

\begin{lemma}\label{lemmabelow}  Under the assumption \eqref{H5}, the killing term $
    \kappa(x)$ given in the right side of the dual operator $L^*$ defined by \eqref{du123}
exists such that $\kappa(x)\,dx$ is the vague limit of the sequence of (signed) measures $\big\{\big(-2\int_{|x-y|>1/m} j_a(x,y)\,dy\big)\,dx\big\}_{m\in\N}$.
\end{lemma}
\begin{proof}   In the proof of Theorem \ref{dual} we have seen that the limit
 $$\mu(fg):=\lim_{\varepsilon\to0}I_2^\varepsilon=-2\lim_{\varepsilon\to0}\int f(x)g(x) \bigg[\int_{|y-x|\ge \varepsilon} j_a(x,y)\,dy\bigg] dx$$exists for all $f$, $g\in C_c^\infty(\R^n)$. Thus, $\mu$ is an element in $\mathscr{D}'=(C_c^\infty(\R^n))^*$. Since, under \eqref{H5},   $|\int \varphi(x)\,\mu(dx)|\leq c_K \|\varphi\|_\infty$ for all continuous functions $u$ with support in the compact set $K$,  $\mu$ is a distribution of order zero, hence a (signed) Radon measure. Thus, it remains to show that $\mu$ is absolutely continuous with respect to Lebesgue measure $dx$.

For any $\varphi\in C_c^\infty(\R^n)$, we know from \eqref{H5} that for all $x\in \mathrm{supp}\,\varphi$,
\begin{align*}
    f_\varphi(x)
    &:=\limsup_{m\to\infty}\bigg|\int_{|x-y|>1/m}j_a(x,y)\,dy\bigg|\\
    &\leq \sup_{x\in \mathrm{supp}\,\varphi}\sup_{m\in\N} \bigg|\int_{|x-y|>1/m}j_a(x,y)\,dy\bigg|\\
    &\leq C_\varphi
    <\infty.
\end{align*}
Therefore, by (a variant of) Fatou's lemma, we have
\begin{align*} \Big| \int \varphi(x)\,\mu(dx)\Big|&\le \limsup_{m\to\infty}\int |\varphi(x)|\bigg|\int_{|x-y|>1/m}j_a(x,y)\,dy \bigg|\,dx\\
&\le \int |\varphi(x)|\limsup_{m\to\infty}\bigg|\int_{|x-y|>1/m}j_a(x,y)\,dy \bigg|\,dx\\
&\le \int |\varphi(x)| f_\varphi(x)\,dx .\end{align*}
This proves that $\mu(dx)\ll dx$.
\end{proof}

\medskip

We have seen in Theorem \ref{dual} that, under the assumptions \eqref{H1}--\eqref{H5}, the operator $(L, C_c^\infty(\R^n))$ given by \eqref{du1} generates a semi-Dirichlet form $(\eta, \mathscr{F}^0\times \mathscr{F}^0)$ on $L^2(\R^n)$. Therefore, there exists a unique sub-Markov semigroup $\{T_t\}_{t\ge0}$ associated with the form $(\eta, \mathscr{F}^0\times \mathscr{F}^0)$ on $L^2(\R^n)$. The dual semigroup $\{{T}^*_t\}_{t\ge0}$ is positivity preserving but it is, in general, not sub-Markovian, see \cite[Section 1]{FU} and the references therein. The following result provides a sufficient condition for the sub-Markov property of the dual semigroup.

\begin{corollary}
    Let $L, \Lambda$ and $L^*$ be as in Theorem \ref{dual}. If the killing term $\kappa(x)$ in the representation \eqref{du123} of the dual operator $L^*$ is non-positive, then the dual semigroup $\{{T}^*_t\}_{t\ge0}$ corresponding to the semi-Dirichlet form $(\eta, \mathscr{F}^0\times \mathscr{F}^0)$ is sub-Markovian.
\end{corollary}

\begin{proof}
    We have seen in Theorem \ref{dual} that the operator $(\Lambda, C_c^\infty(\R^n))$ given by \eqref{du12300} is the pre-generator of the semi-Dirichlet form $(\eta_\Lambda, \mathscr{F}^0\times \mathscr{F}^0)$ generated by the kernel $j(y,x)$  on $L^2(\R^n)$, cf.\ \eqref{pregenerator} with $D(\Lambda)=C_c^\infty(\R^n)$. Therefore, there is a unique sub-Markov semigroup $\{T^\Lambda_t\}_{t\ge0}$ associated with the form $(\eta_\Lambda, \mathscr{F}^0\times \mathscr{F}^0)$ on $L^2(\R^n)$. It is well-known, cf.\ for example \cite[Corollary, page 334]{kunita}, that $\{T_t^\Lambda\}_{t\geq 0}$ is sub-Markovian if, and only if, the corresponding form satisfies
    $$
        Nu\in\mathscr F^0
        \quad\text{and}\quad
        \eta_\Lambda(Nu, u-Nu) \geq 0
        \quad\text{for all $u\in \mathscr F^0$}
    $$
    where $Nu = u^+\wedge 1$ is the normal contraction.
    By definition, the form corresponding to $\{{T}^*_t\}_{t\ge0}$ is $\eta_{L^*}(u,v) = \eta_L(v,u)$, and the proof of Theorem \ref{dual} shows that $L^*u = \Lambda u + \kappa u$ and $\eta_{L^*}(u,v) = \eta_\Lambda(u,v) - \langle\kappa u, v\rangle_{L^2}$. Thus,
    \begin{align*}
        \eta_{L^*}(Nu,u-Nu)
        &= \eta_\Lambda(Nu,u - Nu) - \langle\kappa Nu, u-Nu\rangle_{L^2}\\
        &\geq -\langle\kappa Nu, u-Nu\rangle_{L^2}\\
        &= -\int \kappa(x) (u^+(x)\wedge 1) (u(x)-1)^+\,dx
        \geq 0.
    \end{align*}
    Therefore, $\{T_t^*\}_{t\geq 0}$ is sub-Markovian.
\end{proof}

\section{Example: Stable-like Processes}

Let $E=\R^n$, and $m(dx)=dx$ be Lebesgue measure on $\R^n$. Consider the following integro-differential operator
$$
    Lu(x)= w(x)\int_{z\neq0}\Big(u(x+z)-u(x)-\nabla u(x)\cdot z\I_{\{|z|\le 1\}}(z)\Big)|z|^{-n-\alpha(x)}\,dz
$$
for $u\in C_c^\infty(\R^n)$. The weight function $w(x)$ is chosen in such a way that
$$
    w(x)
    = \alpha(x)2^{\alpha(x)-1}\, \frac{\Gamma\big(\frac 12\alpha(x)+\frac 12 n\big)}{\pi^{n/2}\,\Gamma\big(1-\frac 12\alpha(x)\big)},
$$
and so $L e_\xi(x)=-|\xi|^{\alpha(x)}e_\xi(x)$, where $e_\xi(x)=e^{ix\cdot\xi}$,
see e.g.\ \cite[Exercise 18.23, page 184]{ber-for}. With this norming, $L$ can be written as a pseudo-differential operator $-p(x,D)$ with the symbol $-|\xi|^{\alpha(x)}$,
$$
    Lu(x) = \int e^{ix\cdot\xi} |\xi|^{\alpha(x)}\widehat u(\xi)\,d\xi = -(-\Delta)^{\alpha(x)} u(x),
$$
and this shows that $L=-(-\Delta)^{\alpha(x)}$ is a stable-like operator in the sense of Bass \cite{B}.
Note that, the stable-like operator is a special case of the L\'{e}vy type operator given by \eqref{du1}, and in this case, since for any $x$, $z\in\R^d$, $$j(x,x+z)=w(x)|z|^{-n-\alpha(x)},$$ the second term on the right hand side of \eqref{du1} vanishes automatically.

For $r>0$, define
$$
    \beta(r):=\sup_{|x-y|\le r} |\alpha(x)-\alpha(y)|.
$$
%In this setting, Theorems \ref{th1} and \ref{dual} show:

\begin{proposition}\label{stable}
Let $L=-(-\Delta)^{\alpha(x)}$ and suppose that there exist $0<\alpha_1\le \alpha_2<2$ such that
$$
    \alpha_1\le \alpha(x)\le \alpha_2
    \quad\text{for all}\quad x\in \R^n,
$$
and
$$
    \int_0^1 \frac{\big(\beta(r) |\log r|\big)^2}{r^{1+\alpha_2}}\,dr < \infty.
$$

\noindent\textup{(i)}
The operator $(L, C_c^\infty(\R^n))$ generates a regular lower bounded semi-Dirichlet form on $L^2(\R^n)$ associated with the kernel
$$
    k(x,y)=w(x)|x-y|^{-n-\alpha(x)}.
$$

\smallskip\noindent\textup{(ii)}
If for all compact sets $K\subset\R^n$
\begin{equation}\label{weaklimit}
   \sup_{x\in K} \sup_{\epsilon > 0}\bigg|\int_{|z|\ge \varepsilon}\bigg(\frac{w(x+z)}{|z|^{n+\alpha(x+z)}}-\frac{w(x)}{|z|^{n+\alpha(x)}}\bigg)\,dz\bigg|
   \leq c_K <\infty,
\end{equation}
then the formal adjoint of $L$ is given by
\begin{align*}
    L^*f(x)
    &=\int_{z\neq0}\Big(f(x+z)-f(x)-\nabla f(x)\cdot z{\I }_{\{|z|\le1\}}\Big)\frac{w(x+z)}{|z|^{n+\alpha(x+z)}}\,dz\\
    &\quad+\frac{1}{2}\int_{0<|z|\le 1} z \bigg(\frac{w(x+z)}{|z|^{n+\alpha(x+z)}}-\frac{w(x-z)}{|z|^{n+\alpha(x-z)}}\bigg)\,dz\cdot\nabla f(x)
    +\kappa(x)f(x).
\end{align*}
where $\kappa(x)$ is the density of a (signed) Radon measure on $\R^n$ such that $\kappa(x)\,dx$ is the vague limit of the sequence of (signed) measures $$\bigg\{\bigg(\int_{|z|>1/m} \bigg(\frac{w(x+z)}{|z|^{n+\alpha(x+z)}}-\frac{w(x)}{|z|^{n+\alpha(x)}}\bigg)\,dz\bigg)\,dx\bigg\}_{m\in\N}.$$
\end{proposition}
\begin{proof}
We check the conditions of Theorems \ref{th1} and \ref{dual}. Set $$k(x,y)=w(x)|x-y|^{-n-\alpha(x)}.$$ Then,
\begin{align*}
    k_s(x,y) &= \frac{1}{2}\Big(w(x)|x-y|^{-n-\alpha(x)}+w(y)|x-y|^{-n-\alpha(y)}\Big),\\
    k_a(x,y) &= \frac{1}{2}\Big(w(x)|x-y|^{-n-\alpha(x)}-w(y)|x-y|^{-n-\alpha(y)}\Big).
\end{align*}
From the definition of $w(x)$ it is easy to see that there exist constants $c_j>0, j=1,2,3$, such that for any $x,y\in\R^n$,
$$
    c_1\le w(x)\le c_2, \qquad |w(x)-w(y)|\le c_3|\alpha(x)-\alpha(y)|,
$$
see \cite[proof of Proposition 5.1]{FU}.

(i) Since \eqref{cond1} is obviously satisfied, we only have to verify \eqref{cond2} of Theorem \ref{th1}. We have
$$
    \sup_{x\in\R^n} \int_{|x-y|\ge 1} \frac{k_a^2(x,y)}{k_s(x,y)}\,dy
    \le \sup_{x\in\R^n} \int_{|x-y|\ge 1}  k_s(x,y)\,dy
    \le c\int_1^\infty r^{-1-\alpha_1}\,dr<\infty.
$$
To see
$$
    \sup_{x\in\R^n} \int_{|x-y|\le 1} \frac{k_a^2(x,y)}{k_s(x,y)}\,dy<\infty
$$
we write
$$
    k_a(x,y)
    =\frac{1}{2}\Big[(w(x)-w(y))|x-y|^{-n-\alpha(x)}+w(y)|x-y|^{-n}\big(|x-y|^{-\alpha(x)}-|x-y|^{-\alpha(y)}\big)\Big].
$$
Then
\begin{align*}
    \int_{|x-y|\le 1} &\frac{(w(x)-w(y))^2\,(|x-y|^{-n-\alpha(x)})^2}{w(x)|x-y|^{-n-\alpha(x)}+w(y)|x-y|^{-n-\alpha(y)}}\,dy\\
    &\le c \int_{|x-y|\le 1} \frac{(\alpha(x)-\alpha(y))^2}{|x-y|^{n+\alpha(x)}}\,dy\\
    &\le c\int_{|x-y|\le 1}\frac{\beta^2(|x-y|)}{|x-y|^{n+\alpha_2}} \,dy\\
    &= c\int_0^1 \frac{\beta^2(r)}{r^{1+\alpha_2}}\,dr.
\end{align*}
Since
$$
    |x-y|^{-\alpha(x)}-|x-y|^{-\alpha(y)}
    =\int_{\alpha(y)}^{\alpha(x)}|x-y|^{-u}\log |x-y|^{-1} \,du,
$$
we obtain for all $|x-y|\le 1$,
$$
    \big(|x-y|^{-\alpha(x)}-|x-y|^{-\alpha(y)}\big)^2
    \le \big(\log |x-y|^{-1}\big)^2\big(\alpha(x)-\alpha(y)\big)^2|x-y|^{-2(\alpha(x)\vee \alpha(y))}.
$$
Therefore,
\begin{align*}
    &\int_{|x-y|\le 1} \frac{w(y)^2\,|x-y|^{-2n}\,
    \big(|x-y|^{-\alpha(x)}-|x-y|^{-\alpha(y)}\big)^2}{w(x)|x-y|^{-n-\alpha(x)} +w(y)|x-y|^{-n-\alpha(y)}}\,dy\\
    &\le c \int_{|x-y|\le 1} \frac{\big(\alpha(x)-\alpha(y)\big)^2 \big(\log|x-y|^{-1}\big)^2}{|x-y|^n}
    \frac{|x-y|^{-2(\alpha(x)\vee \alpha(y))}}{w(x)|x-y|^{-\alpha(x)}+w(y)|x-y|^{-\alpha(y)}}\,dy\\
    &\le c\int_{|x-y|\le 1} \frac{\big(\alpha(x)-\alpha(y)\big)^2 \big(\log|x-y|^{-1}\big)^2} {|x-y|^n} \,|x-y|^{-(\alpha(x)\vee \alpha(y))}\,dy\\
    &\le c\int_{|x-y|\le 1} \frac{\beta^2(|x-y|)\big(\log|x-y|^{-1}\big)^2}{|x-y|^{n+\alpha_2}}\,dy\\
    &\le c\int_0^1 \frac{(\beta(r)|\log r|)^2}{r^{1+\alpha_2}}\,dr,
\end{align*}
%where in the second inequality we have used the fact that $$|x-y|^{-(\alpha(x)\vee \alpha(y))}\le c\bigg(w(x)|x-y|^{-\alpha(x)}+w(y)|x-y|^{-\alpha(y)}\bigg).$$
and \eqref{cond2} follows.

\smallskip\noindent\textup{(ii)} Clearly, the conditions \eqref{H1} and \eqref{H2} are satisfied. From part (i), we know that \eqref{H4} is also valid. Since \eqref{H5} is covered by \eqref{weaklimit}, we only have to verify \eqref{H3}. For all $x\in\R^n$,
\begin{align*}
    \int_{|x-y|\le 1} & |x-y| \, \bigg|\frac{w(y)}{|x-y|^{n+\alpha(y)}}-\frac{w(x)}{|x-y|^{n+\alpha(x)}}\bigg|\,dy\\
    &\le \int_{|x-y|\le 1} |x-y|\, |w(x)-w(y)|\, |x-y|^{-n-\alpha(y)}\,dy\\
    &\quad+ \int_{|x-y|\le 1} |x-y|\, w(x)\, \big||x-y|^{-n-\alpha(y)}- |x-y|^{-n-\alpha(y)}\big|\,dy.
\end{align*}
With similar arguments as in the proof of part (i) we see that the right hand side of the inequality above is smaller than
$$
    c\int_0^1 \frac{\beta(r) (1+|\log r|)}{r^{\alpha_2}}\,dr
    \leq c'\left(\int_0^1 \frac{\beta(r) \,|\log r|}{r^{\alpha_2}}\,dr + 1\right).
$$
Pick $\gamma < 1/2$ such that $\alpha_2 \leq 1+2\gamma < 2$. By the Cauchy-Schwarz inequality, we find
\begin{align*}
    \int_0^1 \frac{\beta(r) |\log r|}{r^{\alpha_2}}\,dr
%    &= \int_0^1 \frac 1{r^\gamma}\,\frac{\beta(r) |\log r|}{r^{\alpha_2-\gamma}}\,dr\\
    &\leq \left(\int_0^1 \frac 1{r^{2\gamma}}\,dr\right)^{1/2}
    \left(\int_0^1 \frac{\beta(r)^2 |\log r|^2}{r^{2\alpha_2-2\gamma}}\,dr\right)^{1/2}\\
%    &\leq \frac 1{\sqrt{1-2\gamma}} \left(\int_0^1 \frac{\beta(r)^2 |\log r|^2}{r^{2\alpha_2-2\gamma}}\,dr\right)^{1/2}\\
    &\leq \frac 1{\sqrt{1-2\gamma}} \left(\int_0^1 \frac{\beta(r)^2 |\log r|^2}{r^{1+\alpha_2}}\,dr\right)^{1/2}.
\qedhere
\end{align*}

\end{proof}

We close with this section with some comments on related results in \cite{FU,U1} and our Proposition \ref{stable}.
\begin{remark}(i) Assume that for $r\to 0$
$$
    \beta(r)\asymp r^\beta,\quad \beta>{\alpha_2}/{2},
\qquad
\text{or}
\qquad
    \beta(r)\asymp r^{\alpha_2/2} |\log r|^\varepsilon, \quad \varepsilon<-3/2.
$$
Then Proposition \ref{stable}(i) applies and shows that the operator $(L, C_c^\infty(\R^n))$ generates a regular
lower bounded semi-Dirichlet form on $L^2(\R^n)$. This is, in particular, the case if the index function $\alpha(x)$ is locally Lipschitz continuous. Thus, Proposition \ref{stable}(i) improves \cite[Proposition 5.1]{FU} where the following assumptions are used: \emph{there exist positive constants $\alpha_1$, $\alpha_2$, $M$ and $\delta$ such that for $x$, $y\in\R^n$,
$$
    0<\alpha_1\le \alpha(x)\le \alpha_2<2
    \quad\text{with}\quad \alpha_2< 1+\alpha_1/2
$$
and
$$
    |\alpha(x)-\alpha(y)|\le M|x-y|^\delta
    \quad\text{with}\quad
    0< \frac{1}{2}\big(2\alpha_2-\alpha_1\big)< \delta\le 1.
$$}

(ii) With essentially the same calculations as in the proof of Proposition \ref{stable}(ii) we can see that
$$
    \int_0^1 \frac{\beta(r) |\log r|}{r^{1+\alpha_2}}\,dr<\infty
$$
guarantees that
$$
    \int\bigg(\frac{w(x+z)}{|z|^{n+\alpha(x+z)}}-\frac{w(x)}{|z|^{n+\alpha(x)}}\bigg)dz
    \quad\text{exists}
$$
as a Lebesgue integral. This is, for example, the case if \emph{there exist positive constants $\alpha_1,\alpha_2, \delta$ and $M$ such that $0<\alpha_1\le \alpha_2<1$, $\delta\in(\alpha_2,1]$, and for all $x,y\in\R^n$,
$$
    \alpha_1\le \alpha(x)\le \alpha_2
    \quad\text{and}\quad
    |\alpha(x)-\alpha(y)|\le M|x-y|^\delta.
$$}
This shows that Proposition \ref{stable}(ii) covers the conclusion in \cite[Remark 4]{U1}.

(iii) The assumption \eqref{weaklimit} ensures the existence of the killing term $\kappa(x)$ in the dual operator $L^*$ given in Proposition \ref{stable}(ii). We will claim that \emph{under the conditions of Proposition \ref{stable} and if the index function $\alpha(x)$ also belongs to $C_b^2(\R^n)$, then \eqref{weaklimit} is fulfilled}. Indeed, in this case it is easily seen from the definition of $w(x)$ that $\alpha\in C_b^2(\R^n)$ entails $w\in C_b^2(\R^n)$. For any fixed $z\in\R^n$ with $z\neq0$, set $f_z(x):=w(x)|z|^{-n-\alpha(x)}$. Then, for any $x$, $z\in \R^n$,
$$\nabla f_z(x)=|z|^{-n-\alpha(x)}\big[\nabla w(x)-\big(\log |z|\big)w(x)\nabla \alpha(x)\big],$$ and there is a constant $c_1>0$ such that for any $x$, $z\in\R^n$ with $0<|z|<1$,
$$
    \big|\nabla^2 f_z(x)\big|\le c_1  |z|^{-n-\alpha_2}\big(1+\log^2 |z|\big).
$$
Thus, by Taylor's formula there exists $\theta\in(0,1)$ such that
$$
    f_z(x+y)-f_z(x)=\nabla f_z(x)\cdot y+ \frac{1}{2} y^\top\nabla^2 f_z(x+\theta y)y
    \qquad
    \text{for all\ } x,y,z\in\R^n.
$$
Therefore, for any $x\in\R^n$, by setting $y=z$ in the equality above,
\begin{align*}&\sup_{\epsilon > 0}\bigg|\int_{|z|\ge \varepsilon}\Big(\frac{w(x+z)}{|z|^{n+\alpha(x+z)}}-\frac{w(x)}{|z|^{n+\alpha(x)}}\Big)\,dz\bigg|\\
&\le \sup_{\epsilon > 0}\bigg|\int_{\varepsilon\le|z|<1 }\big(f_z(x+z)-f_z(x)\big)\,dz\bigg| +\int_{|z|\ge1} f_z(x)\,dz+\int_{|z|\ge 1}f_z(x+z)\,dz\\
&\le c_3+\sup_{\epsilon > 0}\bigg|\int_{\varepsilon\le|z|<1 }\bigg(\nabla f_z(x)\cdot z+ \frac{1}{2} z^\top\nabla^2 f_z(x+\theta z)z\bigg)\,dz\bigg|\\
&= c_3+\frac{1}{2}\sup_{\epsilon > 0}\bigg|\int_{\varepsilon\le|z|<1 }\bigg( z^\top\nabla^2 f_z(x+\theta z)z\bigg)\,dz\bigg|\\
&\le c_3+ c_4 \sup_{\epsilon > 0}\int_{\varepsilon\le|z|<1 }|z|^{-n-\alpha_2+2}\big(1+\log^2 |z|\big)\,dz\\
&\le c_3+c_5\int_0^1\frac{1+\log^2 r}{r^{\alpha_2-1}}\,dr=:c_6<\infty.\end{align*} The desired assertion follows.
\end{remark}

\begin{acknowledgement}
Financial support through DFG (grant Schi 419/5-2) and DAAD (PPP Kroatien) (for Ren\'{e} L.\ Schilling),
the National Natural Science Foundation of China (No.\ 11126350 and 11201073) and the Program of
Excellent Young Talents in Universities
of Fujian (No.\ JA10058 and JA11051) (for Jian Wang)
is gratefully
acknowledged. We would like to thank Professors M.\ Fukushima and T.\ Uemura \ for helpful discussions on an earlier version of this paper. The comments of an anonymous referee helped to improve the presentation.
\end{acknowledgement}

\end{document}